\documentclass{amsart}
\usepackage{amsmath,amssymb,amsthm}
\usepackage{cite}

\numberwithin{equation}{section}

\newtheorem{Theorem}{Theorem}[section]
\newtheorem{Definition}[Theorem]{Definition}
\newtheorem{Corollary}[Theorem]{Corollary}
\newtheorem{Lemma}[Theorem]{Lemma}
\newtheorem{Proposition}[Theorem]{Proposition}

\makeatletter
\def\section{\@startsection {section}{1}{\z@}%
 {2.5ex plus - 1ex minus -.2ex}%
 {.5ex \@plus.3ex}%
 {\large\bfseries}%
}
\def\subsection{\@startsection {subsection}{1}{\z@}%
 {2.5ex plus - 1ex minus -.2ex}%
 {.5ex \@plus.3ex}%
 {\normalsize\bfseries}%
}
\makeatother

\begin{document}
\title[global non-existence of semirelativistic eq. with mass]{
Remark on the global non-existence of
semirelativistic equations with non-gauge invariant power type nonlinearity with mass}

\author[K. Fujiwara]{Kazumasa Fujiwara}

\address{%
Centro di Ricerca Matematica Ennio De Giorgi\\
Scuola Normale Superiore\\
Pisa, Italy.
}

\email{kazumasa.fujiwara@sns.it}

\begin{abstract}
The non-existence of global solutions
for semirelativistic equations with non-gauge invariant power type nonlinearity with mass
is studied in the frame work of weighted $L^1$.
In particular,
a priori control of weighted integral of solutions is obtained
by introducing a pointwise estimate of fractional derivative of some weight functions.
Especially, small data blowup with small mass is obtained.
\end{abstract}

\maketitle

\section{Introduction}

We consider the Cauchy problem for the following
semirelativistic equations with non-gauge invariant power type nonlinearity:
	\begin{align}
	\begin{cases}
	i \partial_t u + (m^2-\Delta)^{1/2} u = \lambda |u|^p,
	&\quad t \in \lbrack 0, T), \quad x \in \mathbb R^n,\\
	u(0) = u_0,
	&\quad x \in \mathbb R^n,
	\end{cases}
	\label{eq:1.1}
	\end{align}
with $m \geq 0$, $\lambda \in \mathbb C \backslash \{0\}$,
where $\partial_t = \partial/\partial t$
and $\Delta$ is the Laplacian in $\mathbb R^n$.
Here $(m^2- \Delta)^{1/2}$ is realized as a Fourier multiplier with symbol
$(m^2+|\xi|^2)^{1/2}$: $(m^2-\Delta)^{1/2} = \mathfrak F^{-1} (m^2+|\xi|^2)^{1/2} \mathfrak F$,
where $\mathfrak F$ is the Fourier transform defined by
	\[
	(\mathfrak F u )(\xi) = \hat u(\xi)
	= (2 \pi)^{-n/2} \int_{\mathbb R^n} u(x) e^{-i x \cdot \xi} dx.
	\]

We remark that the Cauchy problem such as \eqref{eq:1.1} arises
in various physical settings and accordingly,
semirelativistic equations are also called
half-wave equations, fractional Schr\"odinger equations, and so on,
see \cite{bib:3,bib:15,bib:16} and reference therein.

The local existence for \eqref{eq:1.1} in the $H^s(\mathbb R^n)$ framework
is easily seen if $s > n / 2$,
where $H^s(\mathbb R^n)$ is the usual Sobolev space
defined by $(1-\Delta)^{-s} L^2(\mathbb R^n)$.
Here the local existence in the $H^s(\mathbb R^n)$ framework means
that for any $H^s(\mathbb R^n)$ initial data,
there is a positive time $T$ such that
there is a solution for the corresponding integral equation,
	\begin{align}
	u(t)
	= e^{it(m^2-\Delta)^{1/2}}u_0
	- i \lambda \int_0^t e^{i(t-t')(m^2-\Delta)^{1/2}} |u(t')|^p dt',
	\label{eq:1.2}
	\end{align}
in $C([0,T);H^s(\mathbb R^n))$.
We remark that for $s > n/2$, local solution for \eqref{eq:1.2}
may be constructed by a standard contraction argument
with the Sobolev embedding $H^s(\mathbb R^n) \hookrightarrow L^\infty(\mathbb R^n)$
which holds if and only if $s > n/2$.
We also remark that in the one dimensional case,
$s > 1/2$ is also the necessary condition for the local existence
in the $H^s(\mathbb R)$ framework
because the non-existence of local weak solutions to \eqref{eq:1.1}
with some $H^{1/2}(\mathbb R)$ data is shown in \cite{bib:8}.
In general setting, the necessary condition is still open
and partial results are discussed in \cite{bib:2,bib:9,bib:13}.
We also remark that in massless case,
\eqref{eq:1.1} is scaling invariant.
Namely, when $u$ is a solution to \eqref{eq:1.1} with initial data $u_0$,
then for any $\rho > 0$, the pair,
	\[
	u_\rho(t,x) = \rho^{1/(p-1)} u(\rho t, \rho x),
	\quad u_{0,\rho} = \rho^{1/(p-1)} u_0(\rho x)
	\]
also satisfies \eqref{eq:1.1}.
Then the case where $(s,q)$ satisfies that for $u_0 \in H_q^s \backslash \{ 0 \}$,
	\[
	\| (-\Delta)^{s/2} u_{0,\rho} \|_{L^q(\mathbb R^n)}
	\to \infty \quad \mathrm{as} \quad \rho \to \infty
	\ \Longleftrightarrow \ 
	s - \frac{n}{q} + \frac{1}{p-1}
	> 0
	\]
is called $H^s_q(\mathbb R^n)$ scaling subcritical case,
where $H_q^s(\mathbb R^n) = (1-\Delta)^{-s/2} L^q(\mathbb R^n)$.
Moreover, if $ s = n/q + 1/(p-1)$,
we call the case as $H^s_q(\mathbb R^n)$ scaling critical case.
In the $H^s(\mathbb R^n)$ scaling subcritical case, in general,
the local existence in $H^s(\mathbb R^n)$ framework
is expected but this is not our case
because the case where $n=1$ and $s=1/2$
is $H^s(\mathbb R)$ scaling subcritical with any $p >1$.

In the present paper, we revisit the global non-existence of \eqref{eq:1.1}.
In order to go back to prior works,
we define weak solutions for \eqref{eq:1.1} and its lifespan.
\begin{Definition}
Let $u_0 \in L^2(\mathbb R^n)$.
We say that $u$ is a weak solution to \eqref{eq:1.1} on $[0,T)$,
if $u$ belongs to
$L_\mathrm{loc}^1(0, T ; L^2(\mathbb R^n))
\cap L_\mathrm{loc}^1(0, T ; L^p(\mathbb R^n))$
and the following identity
	\[
	\int_0^\infty \big( u(t)
	\big| i \partial_t \psi (t) + (m^2- \Delta)^{1/2} \psi (t) \big) dt
	= i (u_0 | \psi (0)) + \lambda \int_0^\infty \big( |u(t)|^p \big|\psi (t) \big) dt
	\]
holds for any
$\psi \in C([0,\infty); H^1(\mathbb R^n)) \cap C^1([0,\infty); L^2(\mathbb R^n))$
satisfying
	\[
	\mathrm{supp}\thinspace \psi \subset [0,T] \times \mathbb R,
	\]
where $(\cdot \mid \cdot)$ is the usual $L^2(\mathbb R^n)$ inner product defined by
	\[
	(f \mid g) = \int_{\mathbb R^n} \overline{f(x)} g(x) dx.
	\]
Moreover we define $T_w$ as
	\[
	T_w
	= \inf\{T > 0 \ ; \ \mbox{There is no weak solutions for \eqref{eq:1.1} on $[0,T)$.}\}.
	\]
\end{Definition}
\noindent

At first,
in $L^1(\mathbb R)$ scaling critical and subcritical massless cases,
the global non-existence is shown in \cite{bib:10}.

\begin{Proposition}[{\cite[Theorem 1.3]{bib:10}}]
\label{Proposition:1.2}
If $n=1$, $m=0$, $1 < p \leq 2$, and $u_0 \in (L^1 \cap L^2) (\mathbb R)$ satisfying that
	\begin{align}
	\mathrm{Re} (\overline \lambda u_0) =0,
	\quad
	- \mathrm{Im} \bigg( \int_{\mathbb R} \overline \lambda u_0(x) dx \bigg) > 0,
	\label{eq:1.3}
	\end{align}
then there is no global weak solution,
namely, if $T$ is big enough,
there is no weak solution on $[0,T)$.
\end{Proposition}
\noindent
Here we remark that the case when $p=2$ is $L^1(\mathbb R)$ scaling critical.

Later, Inui \cite{bib:15} obtained the following global non-existence
in $H^s(\mathbb R^n)$ scaling critical and subcritical cases
for large data with $0 \leq s < n/2$
and in $L^2(\mathbb R^n)$ scaling subcritical massless case for small data:
\begin{Proposition}[{\cite[Theorem 1.2]{bib:15}}]
\label{Proposition:1.3}
Let $ s\geq 0$ and $m \geq 0$.
We assume that $1 < p \leq 1 + 2/(n - 2s)$.
Let $f \in H^s(\mathbb R^n)$ satisfy
	\begin{align}
	\mathrm{Re}(\overline \lambda f ) = 0,
	\quad
	-\mathrm{Im}(\overline \lambda f ) \geq
	\begin{cases}
	|x|^{-k},
	&\quad \mathrm{if} \quad|x| \leq 1,\\
	0,
	&\quad \mathrm{if} \quad|x| > 1,
	\end{cases}
	\label{eq:1.4}
	\end{align}
with $k < n/2 - s (\leq 1/(p - 1))$.
If initial value $u_0$ is given by $\mu f$ with positive constant $\mu$,
then there exists $\mu_0$ such that
there is no global weak solution for $\mu > \mu_0$.
Moreover, for any $\mu \in \lbrack \mu_0 , \infty)$,
$T_w$ is estimate by
	\[
	T_w \leq C \mu^{- \frac{1}{\frac{1}{p-1}-k}}.
	\]
with positive constant $C$ which is independent of $\mu$.
\end{Proposition}

\begin{Proposition}[{\cite[Theorem 1.4]{bib:15}}]
\label{Proposition:1.4}
We assume that $1 < p < 1 + 2/n$, $m=0$.
Let $f \in L^2(\mathbb R^n)$ satisfy
	\begin{align}
	\mathrm{Re}(\overline \lambda f ) = 0,
	\quad
	-\mathrm{Im}(\overline \lambda f ) \geq
	\begin{cases}
	0,
	&\quad \mathrm{if} \quad|x| \leq 1,\\
	|x|^{-k},
	&\quad \mathrm{if} \quad|x| > 1,
	\end{cases}
	\label{eq:1.5}
	\end{align}
with $n/2 < k < 1/(p - 1)$.
If initial value $u_0 (x)$ is given by $\mu f (x)$ with $\mu > 0$,
then there is no global weak solution.
Moreover, there exist $\varepsilon > 0$ and a positive constant $C > 0$ such that
	\[
	T_w \leq
	\begin{cases}
	C \mu^{-\frac{1}{\frac{1}{p-1}-k}},
	&\quad \mathrm{if} \quad 0< \mu < \varepsilon ,\\
	2,
	&\quad \mathrm{if} \quad \mu > \varepsilon.
	\end{cases}
	\]
\end{Proposition}

\noindent
We remark that for $0 < s < n/2$,
there are $H^s(\mathbb R^n)$ functions satisfying \eqref{eq:1.4}.
For details, see \cite[Example 5.1]{bib:14}.

In \cite{bib:10,bib:15},
the non-existence of weak solutions are shown by a test function method
introduced by Baras-Pierre \cite{bib:1} and Zhang \cite{bib:17,bib:18}.
In the classical test function argument,
the classical Leibniz rule plays a critical role.
On the other hand,
the fractional derivative $(m^2 - \Delta)^{1/2}$ of compact supported functions
is not controlled pointwisely like classical derivative.
Indeed, since $(m^2-\Delta)^{1/2}$ is non-local,
$\mathrm{supp} \thinspace (m^2-\Delta)^{1/2} \phi$
is bigger than $\mathrm{supp} \thinspace \phi$
for $\phi \in C_c^{\infty}(\mathbb R^n)$ in general,
where $C_c^\infty(\mathbb R^n)$ denotes the collection
of smooth compactly supported functions.
Therefore,
it is impossible to have the following pointwise estimate:
There exists a positive constant $C$ such that
for any $\phi \in C_c^{\infty}(\mathbb R^n)$,
	\begin{align}
	|((m^2-\Delta)^{1/2} \phi^\ell)(x)|
	\leq C |\phi^{\ell-1}(x) ((m^2-\Delta)^{1/2} \phi)(x)|,
	\quad \forall x \in \mathbb R^n
	\label{eq:1.6}
	\end{align}
with $\ell > 1$.
In order to avoid from the difficulty of nonlocality,
in \cite{bib:10,bib:15},
\eqref{eq:1.1} is transformed into
	\begin{align}
	\partial_t^2 v + m^2 v - \Delta v
	= - |\lambda |^2 \partial_t |u|^p,
	\label{eq:1.7}
	\end{align}
where $v = \mathrm{Im}(\overline \lambda u)$.
\eqref{eq:1.7} may be obtained by
applying $- \mathrm{Im} ( \overline \lambda (i \partial t - (m^2-\Delta)^{1/2}))$
to both sides of \eqref{eq:1.1}.
Propositions above were obtained by applying test function method
to \eqref{eq:1.7} with some special test functions.
Here we remark that test function method is relatively indirect method.
Especially, it is impossible to see the behavior of blowup solution
with test function method
because the lifespan is obtained by comparison between initial data and scaling parameter.

On the other hand, in \cite{bib:7},
the global nonexistence of \eqref{eq:1.1} was studied in a more direct manner.

\begin{Proposition}[{\cite[Proposition 4]{bib:7}}]
\label{Proposition:1.5}
Let $m=0$.
Let
	\[
	X(T)
	= C([0,T);L^2(\mathbb R^n))
	\cap C^1([0,T),H^{-1}(\mathbb R^n))
	\cap L^\infty(0,T;L^{p}(\mathbb R^n)).
	\]
Let $u_0 \in L^2(\mathbb R^n)$ satisfy that
	\begin{align}
	M_R(0)
	> C_{n,p,\alpha} R^{n-1/(p-1)},
	\label{eq:1.8}
	\end{align}
with some $R >0$ and $\alpha \in \mathbb C$ satisfying that
	\begin{align}
	\mathrm{Re} ( \alpha \lambda) > 0.
	\label{eq:1.9}
	\end{align}
Here $M_R(0)$ and $C_{n,p,\alpha}$ is given by
	\begin{align*}
	M_R(0)
	&= - \mathrm{Im}
	\bigg( \alpha \int_{\mathbb R^n} u_0(x) \langle x / R \rangle^{-n-1} dx \bigg),\\
	C_{n,p,\alpha}^p
	&=
	2^{1+p'/p} p^{-p'/p} p'^{-1} \mathrm{Re}(\alpha \lambda)^{-p'} | \alpha|^{p+p'}
	A_{n,n+1}^{p'}
	\bigg(\int_{\mathbb R^n} \langle x \rangle^{-n-1} dx \bigg)^p
	\end{align*}
and constant $A_{n,n+1}$ is determined below.
Then there is no solution for \eqref{eq:1.1} in $X(T)$
with $u(0) = u_0$ and $T > T_{n,p,\lambda,\alpha,R}$, where
	\begin{align*}
	T_{n,p,\lambda,\alpha,R}
	&= (p-1)^{-1} D_{n,p,\lambda,\alpha}^{-1}
	R^{n(p-1)}( M_R(0) - C_{n,p,\alpha} R^{n-1/(p-1)} )^{-p+1},\\
	D_{n,p,\lambda,\alpha}
	&= 2^{-1} \mathrm{Re}(\alpha \lambda) |\alpha|^{-p}
	\bigg( \int_{\mathbb R^n} \langle x \rangle^{-n-1} dx \bigg)^{-p+1}.
	\end{align*}
\end{Proposition}
\noindent
We remark that
in the subcritical massless case,
Propositions \ref{Proposition:1.2}, \ref{Proposition:1.3}, and \ref{Proposition:1.4}
may be obtained as corollaries of Proposition \ref{Proposition:1.4}.
Especially, by \eqref{eq:1.9},
conditions \eqref{eq:1.3}, \eqref{eq:1.4}, and \eqref{eq:1.5}
may be relaxed.
For details, see Corollaries 1, 2, and 3 in \cite{bib:7}
and also Corollaries \ref{Corollary:1.10}, \ref{Corollary:1.11}, and \ref{Corollary:1.12}
below.

Proposition \ref{Proposition:1.5} may be obtained
by a modification of test function method of \cite{bib:11}.
Particularly, one can show that,
for solution $u$ to \eqref{eq:1.1},
	\[
	M_R(t)
	= - \mathrm{Im}
	\bigg( \alpha \int_{\mathbb R^n} u(t,x) \langle x / R \rangle^{-n-1} dx \bigg)
	\]
satisfies the ordinary differential inequality,
	\begin{align}
	\frac{d}{dt} (M_R(t) - C_1) \geq C_2 (M_R(t) - C_1)^p
	\label{eq:1.10}
	\end{align}
with some positive constants $C_1$ and $C_2$.
Since a priori weight $L^1$ control of blowup solutions \eqref{eq:1.10} is given,
the approach of \cite{bib:11} may be regarded as relatively direct
comparing to test function methods of \cite{bib:10,bib:15}.
In order to show \eqref{eq:1.10},
again, pointwise control of wight functions like \eqref{eq:1.6} is required.
Since \eqref{eq:1.6} fails for general compactly supported functions,
we consider the estimate of  weight functions decaying polynomially
and obtain the following:

\begin{Lemma}
\label{Lemma:1.6}
Let $\langle x \rangle = ( 1 + |x|^2)^{1/2}$.
For $q > 0$,
there exists a positive constant $A_{n,q}$ depending only on $n$ and $q$
such that for any $x \in \mathbb R^n$,
	\[
	|( (-\Delta)^{1/2} \langle \cdot \rangle^{-q} ) (x) |
	\leq
	\begin{cases}
	A_{n,q} \langle x \rangle^{-q-1},
	&\quad \mathrm{if} \quad 0 < q < n,\\
	A_{n,q}\langle x \rangle^{-n-1} (1+\log (1+|x|)),
	&\quad \mathrm{if} \quad q = n,\\
	A_{n,q} \langle x \rangle^{-n-1},
	&\quad \mathrm{if} \quad q > n.
	\end{cases}
	\]
\end{Lemma}

Lemma \ref{Lemma:1.6} may be shown by a direct computation
with the following representation:
	\begin{align}
	((-\Delta)^{1/2} f)(x)
	= B_{n,s} \thinspace \lim_{\varepsilon \searrow 0}
	\int_{|y| \geq \varepsilon} \frac{f(x) - f(x+y)}{|y|^{n+1}} dy,
	\label{eq:1.11}
	\end{align}
where
	\[
	B_{n,s}
	= \bigg( \int_{\mathbb R^n} \frac{1-\cos(\xi_1)}{|\xi|^{n+1}} d \xi \bigg)^{-1}.
	\]
For details of this representation,
for example, we refer the reader \cite{bib:6}.
If one regards $(-\Delta)^{1/2}$ as $\nabla$,
Lemma \ref{Lemma:1.6} seems natural at least for $0 < q < n$.
When $q \geq n $, the decay rate of fractional derivative is worse
than the expectation form the classical first derivative
but it is sufficient to prove Proposition \ref{Proposition:1.5} and actually sharp.
For details, see Remarks 1 and 2 in Section 2 of \cite{bib:7}.

We also remark that C\'ordoba and C\'ordoba \cite{bib:4} showed that
	\begin{align}
	(-\Delta)^{s/2} (\phi^2 )(x) \leq 2 \phi(x) ((-\Delta)^{s/2} \phi)(x)
	\label{eq:1.12}
	\end{align}
for any $0 \leq s \leq 2$, $\phi \in \mathcal S(\mathbb R^2)$, and $x \in \mathbb R^2$,
where $\mathcal S$ denotes the collection of rapidly decreasing functions.
In general, $\phi \geq 0$ does not imply $(-\Delta)^{s/2} \phi \geq 0$,
and therefore \eqref{eq:1.12} does not imply \eqref{eq:1.6}
even with positive $\phi$.
We also remark that they also used the integral representation of $(-\Delta)^{s/2}$,
which is \eqref{eq:1.11} when $s=1$.
By generalizing \eqref{eq:1.12},
D'Abbicco and Reissig \cite{bib:5} studied global non-existence for
structural damped wave equation possessing fractional derivative.
For the study of structural damped wave equation,
\eqref{eq:1.12} works well because we have non-negative solutions(\cite[Lemma 1]{bib:5}),
which we cannot expect for \eqref{eq:1.1}.

The aim of this paper is
to generalize Proposition \ref{Proposition:1.5}
by introducing the following pointwise estimate:
	\begin{align}
	|( (m^2-\Delta)^{1/2} \langle \cdot \rangle^{-n-1} ) (x) |
	\leq C \langle x \rangle^{-n-1}
	\label{eq:1.13}
	\end{align}
for any $x \in \mathbb R^n$ with some positive constant $C$.

The difficulty to study \eqref{eq:1.13}
is the non-existence of integral representation of $(m^2-\Delta)^{1/2}$
like \eqref{eq:1.11}.
Therefore, we divide our operator into two parts as follows:
	\[
	(m^2 - \Delta)^{1/2}
	= (- \Delta)^{1/2} + \mathcal R,
	\]
where $\mathcal R$ is a Fourier multiplier with the following symbol:
	\[
	(m^2 + |\xi|^2)^{1/2} - |\xi|
	= \int_0^m (\theta^2 + |\xi|^2)^{-1/2} \theta d \theta.
	\]
Thanks to Lemma \ref{Lemma:1.6},
it is sufficient to show the pointwise control of $\mathcal R$.
Fortunately, $\mathcal R$ consists of Bessel potential
and the Bessel potential $(1-\Delta)^{-1/2}$ has an integral kernel $K$.
In particular, we have the following:
\begin{Proposition}[{\cite[Proposition 1.2.5]{bib:12}}]
\label{Proposition:1.7}
Let $K$ be a mesurable function satisfying
	\[
	(1-\Delta)^{-1/2} \phi = K \ast \phi,
	\]
for $\phi \in \mathcal S$, where $\ast$ denotes the convolution.
Then $K$ is strictly positive and $\| K \|_{L^1(\mathbb R^n)}=1$.
Moreover there is a positive constant $\widetilde B_n$ depending only on $n$
and satisfying that
	\begin{align*}
	K(x) &\leq \widetilde B_n e^{-|x|/2},
	&& \mathrm{if} \quad |x| > 2,\\
	K(x) &\leq \widetilde B_n
	\begin{cases}
	\log (\frac{2}{|x|}) + 1 + O(|x|^2),
	& \mathrm{if} \quad n=1,\\
	1 + |x|^{1-n},
	& \mathrm{if} \quad n>1,
	\end{cases}
	&& \mathrm{if} \quad |x| < 2.
	\end{align*}
\end{Proposition}
\noindent
Since $K$ has only integrable singularity at the origin
and decays exponentially,
nonlinear estimate $\mathcal R$ may be obtained by a direct computation.
The next estimate is essential in this paper.

\begin{Proposition}
\label{Proposition:1.8}
For $q > n/2$ and $x \in \mathbb R^n$,
	\begin{align}
	|( (m^2-\Delta)^{1/2} \langle \cdot \rangle^{-q} ) (x) |
	\leq
	|( (-\Delta)^{1/2} \langle \cdot \rangle^{-q} ) (x) |
	+ 2^{q/2} \| \langle \cdot \rangle^q K \|_{L^1(\mathbb R^n)}
	\langle m \rangle^{q+1} \langle x \rangle^{-q}.
	\label{eq:1.14}
	\end{align}
Especially,
	\begin{align}
	|( (m^2-\Delta)^{1/2} \langle \cdot/R \rangle^{-n-1} ) (x) |
	\leq R^{-1} \widetilde A_{n} \langle Rm \rangle^{n+2} \langle x/R \rangle^{-n-1},
	\label{eq:1.15}
	\end{align}
where
$\widetilde A_{n} = A_{n,n+1}
+ 2^{q/2} \| \langle \cdot \rangle^q K \|_{L^1(\mathbb R^n)}$.
\end{Proposition}

Here the condition of $q$ is given to consider the domain of $\mathcal R$
as $L^2(\mathbb R^n)$.
Then by replacing Lemma \ref{Lemma:1.6} by Lemma \ref{Proposition:1.8},
we can generalize Proposition \ref{Proposition:1.5} in case with mass.

\begin{Proposition}
\label{Proposition:1.9}
Let $m \in \mathbb R$.
Let $u_0 \in L^2(\mathbb R^n)$ satisfy that
	\begin{align}
	M_R(0)
	> \widetilde C_{n,p,\alpha} \langle R m \rangle^{(n+2)/(p-1)} R^{n-1/(p-1)},
	\label{eq:1.16}
	\end{align}
with some $R >0$ and $\alpha \in \mathbb C$ satisfying \eqref{eq:1.9},
where $\widetilde C_{n,p,\alpha}$ is given by
	\begin{align*}
	\widetilde C_{n,p,\alpha}^p
	=
	2^{1+p'/p} p^{-p'/p} p'^{-1} \mathrm{Re}(\alpha \lambda)^{-p'} | \alpha|^{p+p'}
	\widetilde A_{n}^{p'}
	\bigg(\int_{\mathbb R^n} \langle x \rangle^{-n-1} dx \bigg)^p.
	\end{align*}
Then there is no solution for \eqref{eq:1.1} in $X(T)$
with $u(0) = u_0$ and $T > \widetilde T_{n,p,m,\lambda,\alpha,R}$, where
	\begin{align*}
	&\widetilde T_{n,p,m,\lambda,\alpha,R}\\
	&= (p-1)^{-1} D_{n,p,\lambda,\alpha}^{-1} R^{n(p-1)}
	( M_R(0)
	- \langle Rm \rangle^{(n+2)/(p-1)}\widetilde C_{n,p,\alpha} R^{n-1/(p-1)} )^{-p+1}.
	\end{align*}
\end{Proposition}

Now, in the subcritical case,
Propositions \ref{Proposition:1.2}, \ref{Proposition:1.3} and \ref{Proposition:1.4}
may be obtained as corollaries of Proposition \ref{Proposition:1.9}.
Here, we remark that since
the Cauchy problem \eqref{eq:1.1} is not scaling invariant essentially,
Propositions \ref{Proposition:1.2} and \ref{Proposition:1.4} seem
difficult to be extended in case of general mass.
However, if mass is sufficiently small,
solutions of \eqref{eq:1.1} are shown to be estimated similarly
to solutions of \eqref{eq:1.1} without mass.

\begin{Corollary}
\label{Corollary:1.10}
Let $1 < p < 1 + 1/n$.
Let $\alpha \in \mathbb C$ and $u_0 \in (L^1 \cap L^2)(\mathbb R^n)$
satisfy \eqref{eq:1.9} and
	\begin{align}
	-\mathrm{Im} \bigg( \alpha \int_{\mathbb R^n} u_0 (x) dx \bigg) > 0.
	\label{eq:1.17}
	\end{align}
Then, for sufficiently small $m$,
there exists no solution in $X(T)$ for sufficiently large $T$.
\end{Corollary}

\begin{Corollary}
\label{Corollary:1.11}
Let $m \in \mathbb R$.
Let $u_0 (x) = \mu f (x)$ where $\mu \gg 1$ and $f$ satisfies
	\begin{align}
	-\mathrm{Im}(\alpha f (x) ) \geq
	\begin{cases}
	|x|^{-k},
	&\quad \mathrm{if} \quad|x| \leq 1,\\
	0,
	&\quad \mathrm{if} \quad|x| > 1,
	\end{cases}
	\label{eq:1.18}
	\end{align}
with some $k < \min(n/2,1/(p-1))$ and $\alpha$ satisfying \eqref{eq:1.9}.
Then there exists some $R_1 > 0$ satisfying \eqref{eq:1.16}
and
	\[
	\widetilde T_{n,p,m,\lambda,\alpha,R_1}
	\leq C \mu ^{-\frac{1}{1/(p-1)-k}}.
	\]
\end{Corollary}

\begin{Corollary}
\label{Corollary:1.12}
Let $u_0 (x) = \mu f (x)$ where $0 < \mu \ll 1$ and $f$ satisfies
	\begin{align}
	-\mathrm{Im}(\alpha f(x) ) \geq
	\begin{cases}
	0,
	&\quad \mathrm{if} \quad|x| \leq 1,\\
	|x|^{-k},
	&\quad \mathrm{if} \quad|x| > 1,
	\end{cases}
	\label{eq:1.19}
	\end{align}
with some $n/2 < k < 1/(p-1)$ and $\alpha$ satisfying \eqref{eq:1.9}.
Then, for sufficiently small $m$,
there exists some $R_2 > 0$ satisfying \eqref{eq:1.16}
and
	\[
	\widetilde T_{n,p,m,\lambda,\alpha,R_2}
	\leq C \mu ^{-\frac{1}{1/(p-1)-\min(n,k)}}.
	\]
\end{Corollary}

We remark that Corollaries \ref{Corollary:1.10}, \ref{Corollary:1.11},
and \ref{Corollary:1.12} correspond to
Propositions \ref{Proposition:1.2}, \ref{Proposition:1.3}, and \ref{Proposition:1.4},
respectively

In the next section, we show Proposition \ref{Proposition:1.8}.
In Section 3,
we show the proof of Proposition \ref{Proposition:1.9}
and Corollaries
\ref{Corollary:1.10},
\ref{Corollary:1.11}, and
\ref{Corollary:1.12}.

\section{Proof of Proposition \ref{Proposition:1.8}}
In order to show \eqref{eq:1.14},
it is sufficient to show for any $q > n/2$,
	\begin{align}
	|\mathcal R \langle \cdot \rangle^{-n-1}(x)|
	\leq 2^{q/2} \| \langle \cdot \rangle^q K \|_{L^1(\mathbb R^n)}
	\langle m \rangle^{q+1}
	\langle x \rangle^{-q}.
	\label{eq:2.1}
	\end{align}
For $\theta > 0$,
	\begin{align*}
	(\theta^2 - \Delta)^{-1/2} \theta f
	&= \theta \mathfrak F^{-1} ( (\theta^2 + |\cdot|^2)^{-1/2} \hat f)\\
	&= \theta^n \mathfrak F^{-1} ( (1 + |\cdot|^2)^{-1/2} (\hat f)_{1/\theta})\\
	&= (1 - \Delta)^{-1/2} f_\theta.
	\end{align*}
Therefore,
	\begin{align}
	\theta (\theta^2 - \Delta)^{-1/2} \langle \cdot \rangle^{-q}
	= (1 - \Delta)^{-1/2} \langle \cdot/\theta \rangle^{-q}
	\leq \langle \theta \rangle^{q} K \ast \langle \cdot \rangle^{-q},
	\label{eq:2.2}
	\end{align}
where we have used the fact that $K$ is positive and for any $x \in \mathbb R^n$,
	\begin{align*}
	\langle x \rangle
	&\leq \langle x/\theta \rangle \langle \theta \rangle.
	\end{align*}
Then by \eqref{eq:2.2},
	\begin{align*}
	\mathcal R \langle \cdot \rangle^{-q} (x)
	&\leq \int_0^m \langle \theta \rangle^{q} d\theta
	\cdot K \ast \langle \cdot \rangle^{-q} (x)\\
	&\leq 2^{q/2} \langle m \rangle^{q+1}
	\int_{\mathbb R^n} K(y) \langle y \rangle^{q} dy
	\cdot \langle x \rangle^{-q},
	\end{align*}
where we have used the fact that,
for any real numbers $x$ and $y$,
	\[
	\langle x \rangle \leq \sqrt 2 \langle x-y \rangle \langle y \rangle.
	\]
This implies \eqref{eq:2.1}.
\eqref{eq:1.15} is shown by the following direct computation:
	\[
	|(m^2-\Delta)^{1/2}\langle \cdot/R \rangle^{-n-1}|
	= R^{-1} |((R^2 m^2-\Delta)^{1/2} \langle \cdot \rangle^{-n-1} )_R|.
	\]

\section{Proof of nonexistence results}
\subsection{Proof of Proposition \ref{Proposition:1.9}}
Proposition \ref{Proposition:1.9} is shown
by the proof of Proposition \ref{Proposition:1.5}
with replacing $A_{n,n+1}$ by $\widetilde A_n \langle Rm \rangle^{n+2}$.
So we omit the detail.

\subsection{Proof of Corollary \ref{Corollary:1.10}}
Let $R_0$ be a positive number satisfying that for any $R > R_0$,
	\[
	M_{R} (0)
	> - \frac 1 2 \mathrm{Im} \bigg( \alpha \int_{\mathbb R^n} u_0(x) dx \bigg).
	\]
We remark that such $R_0$ exists
because of \eqref{eq:1.17} and the Lebesgue dominant theorem.
Moreover, let $R \geq R_0$ be a positive number satisfying that
	\begin{align}
	- \frac 1 2 \mathrm{Im} \bigg( \alpha \int_{\mathbb R^n} u_0(x) dx \bigg)
	> \widetilde C_{n,p,\alpha} 2^{(n+2)/(p-1)} R^{n-1/(p-1)}.
	\label{eq:3.1}
	\end{align}
If $m < R^{-1}$,
then \eqref{eq:3.1} implies \eqref{eq:1.16}
and therefore Proposition \ref{Proposition:1.9} implies Corollary \ref{Corollary:1.10}.

\begin{proof}[Proof of Corollary \ref{Corollary:1.11}]
For $0 < R < 1$, by \eqref{eq:1.18},
	\begin{align*}
	M_R(0)
	&\geq \mu \int_{|x| \leq 1} |x|^{-k} \langle x/R \rangle^{-n-1} dx\\
	&\geq 2^{-n-1} \mu \int_{|x| \leq R} |x|^{-k} dx\\
	&= (n-k)^{-1} 2^{-n-1} \omega_n \mu R^{n-k},
	\end{align*}
where $\omega_n$ is the volume of $S_{n-1}$.
Let $I_1 = (n-k)^{-1} 2^{-n-1} \omega_n$ and
	\[
	R_1 =
	\bigg(
	\frac{\mu I_1}{2^{(n+p+1)/(p-1)} \widetilde C_{n,p,\alpha}} \bigg)^{\frac{1}{k-1/(p-1)}}.
	\]
We put $\mu \gg 1$ so that $R_1 < 1/\max(1,m)$.
Then
	\begin{align*}
	&M_{R_1}(0)
	- \widetilde C_{n,p,\alpha} \langle R_1 m \rangle^{(n+2)/(p-1)} R_1^{n-1/(p-1)}\\
	&\geq R_1^{n-k} ( \mu I_1 - 2^{(n+2)/(p-1)} \widetilde C_{n,p,\alpha} R_1^{k-1/(p-1)})\\
	&\geq 2^{-1} R_1^{n-k} \mu I_1 > 0
	\end{align*}
and therefore \eqref{eq:1.16} is satisfied.
Moreover,
	\begin{align*}
	&\widetilde T_{n,p,m,\lambda,\alpha,R_1}\\
	&\leq (p-1)^{-1} 2^{p-1} D_{n,p,\lambda,\alpha}^{-1}
	\bigg( \frac{\mu I_1}{2^{(n+p+1)/(p-1)} \widetilde C_{n,p,\alpha}}
	\bigg)^{\frac{k(p-1)}{k-1/(p-1)}}
	(\mu I_1)^{-p+1}\\
	&= (p-1)^{-1} 2^{p-1} D_{n,p,\lambda,\alpha}^{-1}
	(2^{(n+p+1)/(p-1)} \widetilde C_{n,p,\alpha})^{-\frac{k(p-1)}{k-1/(p-1)}}
	(\mu I_1)^{-\frac{1}{1/(p-1)-k}}.
	\end{align*}
\end{proof}

\begin{proof}[Proof of Corollary \ref{Corollary:1.12}]
For $R \gg 1$, by \eqref{eq:1.19},
	\begin{align*}
	M_R(0)
	&\geq \mu \int_{|x| \geq 1} |x|^{-k} \langle x/R \rangle^{-n-1} dx\\
	&\geq 2^{-n-1} \mu \int_{1 \leq |x| \leq R} |x|^{-k} dx\\
	&\geq 2^{-n-1} \omega_n \mu \int_1^R r^{n-k-1} dr,\\
	&\geq 2^{-n-1} \omega_n \mu
	\begin{cases}
	(n-k)^{-1} (R^{n-k} - 1),
	&\quad \mathrm{if} \quad k < n,\\
	\int_1^{2} r^{n-k-1} dr,
	&\quad \mathrm{if} \quad k \geq n,
	\end{cases}\\
	&\geq I_2 \mu R^{(n-k)_+},
	\end{align*}
where $(n-k)_+ = \max(n-k,0)$ and
	\[
	I_2
	= \begin{cases}
	2^{-n-2} \omega_n (n-k)^{-1},
	&\quad \mathrm{if} \quad k < n,\\
	2^{-n-1} \omega_n \int_1^{2} r^{n-k-1} dr,
	&\quad \mathrm{if} \quad k \geq n.
	\end{cases}
	\]
Let
	\[
	R_2
	= \bigg( \frac{\mu I_2}{2^{(n+p+1)/(p-1)} \widetilde C_{n,p,\alpha}}
	\bigg)^{\frac{1}{\min(n,k)-1/(p-1)}},
	\]
where $R_2 \gg 1$ if $\mu \ll 1$.
Then, by choosing $m$ so that $m \leq 1/R_2$,
	\begin{align*}
	&M_{R_2}(0)
	- \widetilde C_{n,p,\alpha} \langle R_2 m \rangle^{(n+2)/(p-1)} R_2^{n-1/(p-1)}\\
	&\geq R_2^{(n-k)_+}
	( \mu I_2 - 2^{(n+2)/(p-1)} \widetilde C_{n,p,\alpha} R_2^{\min(n,k)-1/(p-1)})\\
	&\geq 2^{-1} R_2^{(n-k)_+} \mu I_2 > 0
	\end{align*}
and therefore \eqref{eq:1.16} is satisfied.
Moreover,
	\begin{align*}
	&\widetilde T_{n,p,m,\lambda,\alpha,R_2}\\
	&\leq (p-1)^{-1} D_{n,p,\lambda,\alpha}^{-1}
	R_2^{n(p-1)-(n-k)_+(p-1)}
	( \mu I_2 - \widetilde C_{n,p,\alpha} R_2^{\min(n,k) - 1/(p-1)})^{-p+1}\\
	&\leq (p-1)^{-1} 2^{p-1} D_{n,p,\lambda,\alpha}^{-1}
	(2^{(n+p+1)/(p-1)} \widetilde C_{n,p,\alpha})^{-\frac{\min(n,k)(p-1)}{\min(n,k)-1/(p-1)}}
	(\mu I_2)^{-\frac{1}{1/(p-1)-\min(n,k)}}.
	\end{align*}
\end{proof}




\end{document}